\documentclass[11pt,a4paper,reqno]{amsart}

\usepackage{a4wide,color}

\usepackage[english]{babel}
\usepackage[utf8]{inputenc}

\usepackage{bbm,amsmath,amsthm,epsfig,latexsym,marvosym}
\usepackage{amsfonts}
\usepackage{bigints}
\usepackage{amsmath}
\usepackage{amssymb}
\usepackage{subfig}

\allowdisplaybreaks

\usepackage{esint,xfrac}
\usepackage[colorlinks=true,linkcolor=blue,citecolor=red]{hyperref}

\def\Z{\mathbb Z}
\def\R{\mathbb R}
\def\N{\mathbb N}

\def\cal{\mathcal}

\def\F{{\cal F}}

\def\e{\varepsilon}
\def\a{\alpha}

\def\e{\varepsilon}

{\left\lbrace\begin{array}{@{}l@{}}}%
{\end{array}\right.}

\def\d{\, \mathrm{d}}

\def\ca{\mathbbmss{1}}

\def\00{{\bf 0}}

\DeclareMathOperator*{\spt}{spt}
\DeclareMathOperator*{\diam}{diam}

\newtheorem{theorem}{Theorem}[section]

\newtheorem{proposition}[theorem]{Proposition}
\newtheorem{lemma}[theorem]{Lemma}

\theoremstyle{definition}
\newtheorem{remark}[theorem]{Remark}
\newtheorem{definition}[theorem]{Definition}

\numberwithin{equation}{section}
\numberwithin{figure}{section}

\usepackage{xcolor}
\newcommand{\rc}[1]{{\color{black}{#1}}}

\title{A Compactness Theorem for functions on Poisson point clouds}

\author{Marco Caroccia}
\address{Dipartimento di Matematica, Politecnico di Milano}

\email{caroccia.marco@gmail.com}

\begin{document}
\maketitle
\begin{abstract}
    In this work we show a compactness Theorem for discrete functions on Poisson point clouds. We consider sequences with equibounded non-local $p$-Dirichlet energy: the novelty consists in the intermediate-interaction regime at which the non-local energy is computed.  
\end{abstract}
 \section{Introduction}
In this paper we prove a compactness property for sequences of discrete functions defined on Poisson point clouds. We show (Theorem \ref{MainThm:CMP}) that a sequence with uniformly bounded $p$-Dirichlet non-local energy admits a converging subsequence, in a suitable topology (see Definition \ref{def:conv}). Non local energies on discrete systems have been intensively treated, under the variational point of view, in the last two decades, as a way of modelling several phenomena: from fracture mechanics to image denoising and crystallization. Typical discrete variational energies have their natural definition on regular or stochastic lattices (see for instance \cite{alicandro2004general}, \cite{alicandro2011integral}, \cite{alicandro2000finite},  \cite{braides2014discrete}, \cite{braides2002limits}, \cite{braides2018design}, \cite{braides2006discrete}, \cite{Braides2021}, \cite{chambolle1999finite}, \cite{cicalese2020n}, \cite{Friedrich_2020}, \cite{alicandro2014variational}, \cite{alicandro2020variational} and references therein). In some situations we need to account in the model for a randomness component. This leads towards a natural extension of classical models on lattices to a more general \textit{point cloud} framework. Among other interesting features, variational energies on point clouds have been proposed also as a tool to implement supervised or semi-supervised learning algorithms, clustering problem, data denoising and image denoising (as a partial list of literature on the topic we refer the reader to \cite{calder2020rates}, \cite{caroccia2020mumford}, \cite{garciatrillos15aAAA}, \cite{Ruf17a}, \cite{slepcev2019analysis} and references therein). \\
 
\rc{ We here consider point clouds as realizations of a  random variable $\eta$ called Poisson point process (we refer to Subsection \ref{sbsct:PoissonPP} for a detailed definitions). This formalism allows to account for point clouds that are Poisson distributed, in the sense that 
 \[
 \mathbb{P}( \#(points \ in\  A) =m)= \mathrm{Po}(\gamma|A|,m)=e^{-\gamma|A|}\frac{(\gamma|A|)^m}{m!}
 \] 
where $\mathrm{Po}(t,m)$ is the \textit{probability mass function of the Poisson distribution of parameter $t$} (we refer to Definition \ref{sbsct:PoissonPP} for a more accurate Definition of a Poisson point process). The parameter $\gamma$ is called the \textit{intensity of the process}. By introducing a parameter $\e$ we act on the \textit{intensity} of the process, obtaining a sequence of processes $\eta_{\e}$ whose realizations converge (almost surely) to a continuum set  as $\e$ goes to $0$} (in the sense that the normalized discrete measure induced by the process $\eta_{\e}$, weak*-converges to the Lebesgue measure). For any realization $\omega$ of the process, and for $\e>0$, we have a point cloud $\eta_{\e}(\omega)$. For $s_{\e}>0$ the energy at scale $\e$ is then defined for $u:\eta_{\e}(\omega) \rightarrow \R$ on $A\subseteq \R^d$ as:
\begin{equation}\label{eqn;energy}
 \F_{\e}^{\omega}(u;A):=\e^{d}\sum_{x\in \eta_{\e}(\omega)\cap A} \frac{\e^d}{ s_{\e}^d } \sum_{y\in \eta_{\e}(\omega)\cap B_{s_\e}(x)}\left(\frac{|u(y)-u(x)|}{s_{\e}}\right)^p.
 \end{equation}
Notice that the energy can be rephrased as
\[
\F_{\e}^{\omega}(u;A):=\frac{\e^{2d}}{s_{\e}^d}\sum_{\substack{x,y\in \eta_{\e}(\omega):\\ |x-y|\leq s_{\e}}} \left(\frac{|u(y)-u(x)|}{s_{\e}}\right)^p
\]
which express the fact that we are summing up the finite differences of $u$, up to a distance of $s_\e$: the scale of averaging. The scale $\{s_{\e}\}_{\e>0}$ is taken as a sequence $s_{\e}\rightarrow 0$ that, here and in the sequel, will satisfy the decay property:
    \begin{equation}\label{eqn:regime}
    \liminf_{\e\rightarrow 0} \frac{s_\e}{\e (\log(\e^{-d}))^{1/d}}=\beta>0
    \end{equation}
with $\beta<+\infty$. \\

The novelty of this work relies exactly in the assumption on the decay of $s_{\e}$ in terms of $\e$, expressed by \eqref{eqn:regime} with $0<\beta<+\infty$. The typical approach for energies defined on point clouds (as for instance in  \cite{caroccia2020mumford}, \cite{cristoferi18}, \cite{thorpe2019asymptotic}, \cite{trillos2016continuum}) makes use of the hypothesis $\beta=+\infty$ in order to study the asymptotic behavior as $\e\rightarrow 0$. In the long-range regime $\beta=+\infty$, the topology of convergence is the strong one induced by the $TL^1$ distance (see for instance \cite{garciatrillos15} or \cite{trillos2016continuum}). In contrast to the case $\beta=+\infty$ we can consider the case $\beta=0$. In the recent works \cite{carbra2021}, \cite{BraPia20} the authors deal with the short-range regime, where $s_\e\approx M\e$ or where nearest neighbors interactions are the only interactions involved. In the short-range regime it is still possible, at least in the planar case, to deduce the compactness of equibounded sequences (and thus the asymptotic behavior of $\F_{\e}$) but in a much weaker topology than the $TL^1$. In particular the analysis that seems to be missing in literature, and that we aim to fill for what concerns a compactness property, is when $0<\beta<+\infty$. As we state and prove in Theorem \ref{MainThm:CMP}, when $0<\beta<+\infty$, sequences with equibounded energy will be compact in the topology yielding the convergence of Definition \ref{def:conv}: a stronger topology than the one accessible in the $\beta=0$ case but still weaker than the $TL^p$ topology (attainable in the $\beta=+\infty$ case). These differences in the underlying topologies, arising in the asymptotic behavior analysis, are related to the different behavior of the stochastic geometry of the point cloud when looked at different scales. \\

The specialty of the comparison scale $\sigma_{\e}:=\e\log(\e^{-d})^{1/d}$ can be naively explained in this way. Given a stochastic Poisson point process  $\eta_{\e}$, and denoted by $\eta_{\e}(A)$ the number of points of the point cloud that fall in $A\subset \R^d$, the expected value of $\eta_{\e}(A)$  is  (cf. Remark \ref{rmk:exp})
    \[
    \mathbb{E}[\eta_{\e}(A)]\approx \frac{|A|}{\e^d}.
    \]
Thus, for a point $x\in \eta_{\e}$, the expected number of point $y\in \eta_{\e}$ that will interact with $x$ at scale $s_{\e}$ is approximately $\sfrac{s_{\e}^d}{\e^d}$ ($\approx \eta_{\e}(B_{s_{\e}}(x))$). By invoking some well known properties of Poisson distribution it is easy to see that the probability that a point $x$ interacts, at a scale $s_{\e}$, with either too much, or too few points is approximately
 \[
p_{\e}:= \mathbb{P}\left( \eta_{\e}(B_{s_\e}(x))\geq 2\frac{s_{\e}^d}{\e^d} \ \ \ \text{or}\  \ \ \eta_{\e}(B_{s_\e}(x))\leq \frac{1}{2} \frac{s_{\e}^d}{\e^d} \right)\approx e^{-C_d\frac{s_\e^d}{\e^d}}.
 \]
 Therefore depending on the behavior of $s_{\e}$ compared to $\sigma_{\e}$ several phenomena can occur. If $s_{\e}>>\sigma_{\e}$ then $ p_{\e}$ tends to zero fast enough to guarantee that, almost surely, each point $x$ will always interact with approximately $\sfrac{s_{\e}^d}{\e^d}$ points, resulting in no lack of information when approaching the continuum limit as $\e\rightarrow 0$. This is precisely the approach that has been implemented in \cite{caroccia2020mumford}, \cite{garciatrillos15} to investigate energies on point clouds and that allows to deduce convergence in $TL^p$. If $s_{\e}\approx M\e$ then $p_{\e} \approx e^{-C_d M}$ and thus we expect to see a fixed percentage of points that have the wrong number of interactions at scale $M\e$ resulting in a loss of information. In particular, a uniform bound on the energy $\F_{\e}$, might not contain enough information to deduce convergence, up to a subsequence, in a strong topology. Heuristically speaking, the geometry of the point cloud at this scale, becomes closely related to the geometry of perforated domains: all the points with few interactions act as holes in the domain. For the planar case, in \cite{carbra2021} and \cite{BraPia20}, a fine analysis of the geometry of the point clouds, and the identification of regular subclusters, is required in order to describe the asymptotic behavior of $\F_{\e}$. In the intermediate-regime, given by \eqref{eqn:regime}, the probability $p_{\e} \approx \e^{d \beta^d}$ vanishes as $\e\rightarrow 0$, but not fast enough to guarantee that any point $x$ will be well connected to its neighbors at scale $s_{\e}$ (as in the case $\beta=+\infty$). However it is still fast enough to deduce that the number of wrong points remains under control (see Proposition \ref{prop:decay}). This is now similar to the study of energies on perforated domain, where we have some controls on the size of the holes (see for instance \cite{ansini2002asymptotic}, \cite{braides2020homogenization}). In this sense we can deduce convergence, up to subsequences, in a stronger way than in the case $\beta=0$ but in a weaker sense than in the case $\beta=+\infty$.\\
 
The paper is organized as follows. We briefly introduce the main ingredient of our analysis and the main Theorem \ref{MainThm:CMP} in Section \ref{sct:not}. We then proceed to the study of the stochastic geometry of a general realization of the Poisson point process in Section \ref{sct:stogeo}. Technically this analysis is implemented in this way: we divide $\R^d$ in boxes of size $s_{\e}$ and we consider the family of ``bad boxes" that contains either too few or too many points. We show an upper bound on the number of these boxes (Proposition \ref{prop:decay}) and on the volume of each connected component of the union of the boxes (Lemma \ref{lem:DimConn}). Finally, in Section \ref{sct:proof}, we present the proof of our main result \ref{MainThm:CMP}. For $u_{\e}:\eta_{\e}\cap Q\rightarrow \R$ we consider the averaged function
    \[
    u^I_{\e}:=\frac{1}{\eta_{\e}(Q_{s_{\e}}(I))} \sum_{x\in \eta_{\e}\cap Q_{s_{\e}}(I)}u(x)
    \]
where $Q_{s_{\e}}(I)$ does not belong to the set of ``bad boxes". We then have a function
\[
\hat{u}_{\e}(x):=\sum_{\substack{Q_{s_{\e}}(I) \text{ not a} \\ \text{``bad box"}}} u_{\e}^{I}\ca_{Q_{s_{\e}}(I)}(x).
\]
We use the results of Section \ref{sct:stogeo}, on the stochastic geometry of the process, to build an extension operator that extends the function $\hat{u}_{\e}$ also on the ``bad boxes" by keeping under control the finite differences (Lemma \ref{lem:extension}). We then use a compactness Theorem for functions on lattices to prove compactness for the extended functions. \rc{Notice that this argument fails as soon as we loose control on the size of the family  of bad boxes, namely if $\beta=0$. }

\subsection*{Acknowledgment} The author is grateful to Professor Andrea Braides for the many fruitful discussions on the topics here contained. The author is grateful also to the anonymous referee for his/her very precise report and for all the observations that helped enrich the contribution here contained.

\section{Preliminaries and main result}\label{sct:not}
\subsection{General notation} Given $x=(x_1,\ldots,x_n)\in \R^d$, $r>0$ we denote by 
    \[
    Q_r(x):=\left[x_1-\sfrac{r}{2},x_1+\sfrac{r}{2}\right]\times \ldots\times \left[x_n-\sfrac{r}{2},x_n+\sfrac{r}{2}\right]
    \]
the closed cube centered at $x$ and with  side  length $r$. When $r=1$ and  $x_0=0$ we simply write $Q$. For a Borel set $E$, the notation $|E|$ stands for the Lebesgue measure of $E$. For a discrete set $S:=\{a,b,c,\ldots\}$ the notation $\#(S)$ stands for the cardinality of $S$. We will set, for $A\subset \R^d$,
    \begin{align*}
    \mathcal{I}_s(A)&:=\{J\in s\Z^d\cap A\}.
    \end{align*} 
The notation $\ca_A(x)$ stands for the characteristic function of the set $A$. In all the estimates that follow, $C$ will be a constant independent of the crucial quantities under analysis and that might change from line to line.
\subsection{Poisson point process}\label{sbsct:PoissonPP}

Given  a probability space $(\Omega, \mathcal{F},\mathbb{P})$, a point process is a random variable $\eta:\Omega \rightarrow \mathbf{N}_s(\R^d)$, where $\mathbf{N}_s(\R^d)$ is the space of \textit{simple measures} of $\R^d$
\[
\mathbf{N}_s(\R^d):=\left\{\left.\sum_{n\in \N} \delta_{x_n} \ \right| \ \{x_n\}_{n\in \N}\subset \R^d \ \text{any family of distinct points}\  x_n\neq x_k,  \ k\neq n\right\}.
\]
In particular, for any $\omega\in \Omega$, we have that $\eta(\omega)$ is a measure concentrated on a countable family of points. A Poisson point processes, or Poisson point cloud, of intensity $\gamma$ is a point process $\eta:\Omega \rightarrow \mathbf{N}_s(\R^d)$ such that
\begin{itemize}
    \item[a)] for all $k\in\N$ it holds $\mathbb{P}(\eta^{\omega}(A)=k):=e^{-\gamma |A|} \frac{(\gamma |A|)^k}{k!}$;
    \item[b)] for all $B_1,\ldots, B_k$ essentially disjoint borel sets, the random variable $X_j:\Omega\rightarrow \mathbb{N}$, $X_j(\omega):=\eta(\omega)(B_j)$ are independent.
\end{itemize}
For further details on Poisson point processes we refer the reader to \cite{last2017lectures}. Given $\eta$ a Poisson point process we define
    \[
    \eta_{\e}(\omega)(B):=\eta(\omega)(\e^{-1}B).
    \]
Notice that $\eta_{\e}$ is just a Poisson point process with intensity $\gamma \e^{-d}$. Also, for any realization,
\[
\spt(\eta_{\e}(\omega))=\e \spt(\eta(\omega)).
\]
We write, with a slight abuse of notation, $x\in \eta_{\e}(\omega)$, meaning $x\in \spt(\eta_{\e}(\omega))$.
\begin{remark}\label{rmk:exp}
Observe that, for a generic Borel set $A\subset \R^d$,
\begin{align*}
\mathbb{E}[\eta_{\e}(A)]&=\sum_{k=0}^{+\infty} k \mathbb{P}(\eta_{\e}(A)=k)=e^{-\gamma \e^{-d} |A| }\sum_{k=0}^{+\infty} k \frac{\left(\gamma \e^{-d} |A| \right)^{k}}{k!}\\
&=\gamma \e^{-d} |A| e^{-\gamma \e^{-d} |A| }\sum_{k=1}^{+\infty} \frac{\left(\gamma \e^{-d} |A| \right)^{k-1}}{(k-1)!}=\gamma \e^{-d} |A|.
\end{align*}
\end{remark}
\subsection{$p$-Dirichlet energy}\label{sbsct:PDir}
Given a Poisson point cloud $\eta: \Omega\rightarrow \mathbf{N}_s(\R^d)$ with intensity $\gamma$ and $\{s_{\e}\}_{\e>0}$ a sequence converging to zero, we consider the $p$-Dirichlet energy on $A$ of $u:\eta_{\e}(\omega) \rightarrow \R_+$ at scale $s_{\e}$ to be
    \[
    \mathcal{F}^{\omega}_{\e}(u;A):= \e^{d} \sum_{x\in \eta_{\e}(\omega)\cap A} |\mathrm{grad}_{s_{\e}} u(x)|^p
    \]
where
    \[
    |\mathrm{grad}_{s_{\e}} u(x)|^p:=\frac{\e^{d}}{s_{\e}^d}\sum_{y\in \eta_{\e}(\omega) \cap B_{s_{\e}}(x) } \left(\frac{|u(y)-u(x)|}{s_{\e}  }\right)^p.
    \]
In the sequel we omit to specify the dependence of $\F$ and $\eta_{\e}$ on the realization when there is no ambiguity.
    \subsection{Main result}
For $x\in \eta$ we consider the Voronoi cell to be
    \[
    V(x;\eta):= \{y\in \R^d \ | \ |y-x|\leq |y-z| \ \text{for all $z\in \eta \setminus \{x\}$}\}.
    \]
 For a function $u_{\e}:\eta_{\e}\cap Q\rightarrow \R$ we set   $\hat{u}_{\e}:\mathcal{V}_{\e}(\eta_{\e}) \cap Q\rightarrow \R$ 
\[
\hat{u}_{\e}(x):=\sum_{y\in \eta_{\e}\cap Q} u_{\e}(y) \ca_{V(y;\eta_{\e})}(x).
\]
Set
    \[
    \eta^{\alpha} :=\left\{x\in \eta  \ : \ \mathrm{In}(V(x;\eta ))>\alpha, \ \diam(V(x;\eta )<\alpha^{-1}\right\}
    \]
being 
\[
\mathrm{In}(E):=\sup_{r>0}\{\text{there exists $B_r\subset E$}\}.
\]
We refer the reader to Figure \ref{figPPP}. The subcloud $\eta_{\alpha}$ has been firstly introduced in \cite{BraPia20}, where the authors derive the continuum limit of an interface-type energy on Poisson point clouds. \\

For $\a,\e>0$ we set also $\eta^{\alpha}_{\e}:=(\eta_{\e})_{\alpha \e}$, (observe that $\eta^{\alpha}_{\e}= \e\eta^{\alpha}$) and, for a general subcloud of points $\vartheta\subset \eta_{\e}$, we define
\[
\mathcal{V}_{\e}(\vartheta):=\bigcup_{x\in \vartheta} V(x;\eta_\e).
\]
We are in the position to introduce the next Definition.

\begin{figure}
    \centering
    \includegraphics[scale=0.8]{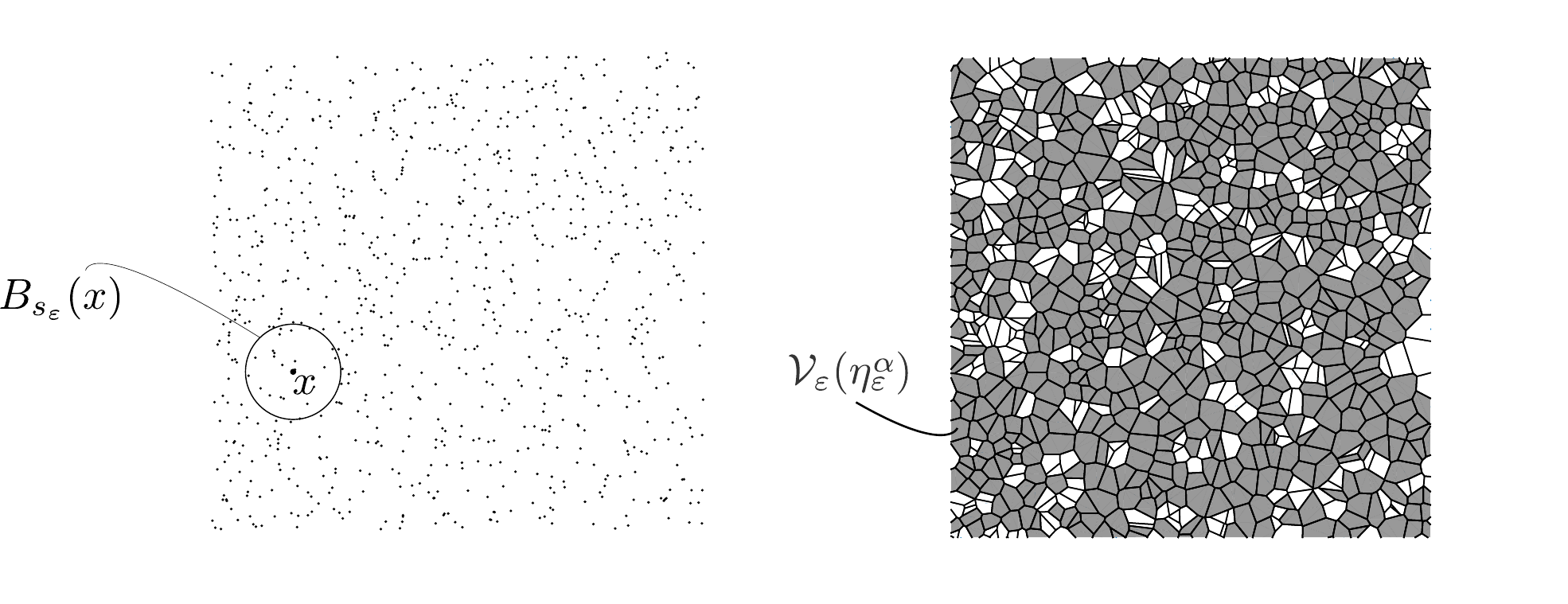}
    \caption{\rc{An illustration of a Poisson point cloud $\eta_{\e}(\omega)$ on the left. On the right  an illustration of  $\mathcal{V}_{\e}(\eta^{\alpha}_{\e})$  (generated by a sub-cluster $\eta^{\alpha}_{\e}$)  depicted in dark grey. }}
    \label{figPPP}
\end{figure}
\begin{definition}[Convergence]\label{def:conv}
We say that a sequence $u_{\e}\stackrel{p}{\longrightarrow} u$ in $A$ if for any $\alpha>0$ it holds
    \[
    \lim_{\e\rightarrow 0}\int_{\mathcal{V}_{\e}(\eta^{\alpha}_{\e})\cap A}|\hat{u}_{\e}-u|^p\d x=0.
    \]

\end{definition}
With this notion of convergence we state our main result.
\begin{theorem}\label{MainThm:CMP}
Let $\{u_{\e}:\eta_{\e} \rightarrow \R\}_{\e>0}$. Let $A\supset Q$ be any open set containing $Q$ and assume that $s_{\e}$ is a sequence satisfying \eqref{eqn:regime}. Then for almost all realization  $\omega\in \Omega$  the following holds. Suppose that
\begin{equation}\label{eqn:bndCMP}
    \sup_{\e>0}\left\{\F_{\e}^{\omega}(u_{\e};A) + \sum_{x\in \eta_{\e}(\omega)\cap Q} u_{\e}(x)^p \e^d\right\}<+\infty.
\end{equation}
Then there exists a subsequence $\{u_{\e_n}\}_{n\in \N}\subset\{u_{\e}\}_{\e>0} $ and a function $u\in W^{1,p}(Q)$ such that $u_{\e_n} \stackrel{q}{\longrightarrow} u$ in $Q$ for all $q<p$, in the sense of Definition \ref{def:conv}.
\end{theorem}
\begin{remark}
Notice that, since the energy is non local, to have compactness on $Q$ we need some control also on the interactions between $x\in Q, y\in Q^c$ with $|x-y|\leq s_{\e}$. This is why we require that the energy bound in \eqref{eqn:bndCMP} to hold on a slightly bigger open set $A\supset Q$. As it will be clear from the proof, Theorem \ref{MainThm:CMP} holds as well if we replace the cube $Q$ with a generic open set.
\end{remark}
\section{Stochastic geometry of the point cloud}\label{sct:stogeo}
For any realization $\omega$, given $s>0$ and $0<\kappa<1$ a parameter we set
    \begin{align}
    \mathcal{I}^{\kappa,\omega}_s(Q)&:=\left\{J\in \mathcal{I}_s(Q) \ : \  \left|\eta_{\e}(\omega)(Q_s(J))-\gamma\frac{s^d}{\e^d}\right|\geq \kappa\gamma\frac{s^d}{\e^d} \right\}\label{eqn:index}\\
    E^\kappa_s(\omega)&:=\bigcup_{  \mathcal{I}^{\kappa, \omega }_s(Q)}  Q_s(J).\label{eqn:set}
       \end{align}
The next Lemma is a standard concentration inequality for random variables (see for instance \cite[Corollary 2.11]{boucheron2013concentration}). Again we will omit the dependence on $\omega$ to lighten the notation.    
\begin{lemma}[Cramer-Chernoff concentration bounds]\label{lemma:CCconcc}
If $X$ is a random variable with Poisson distribution then
    \[
    \mathbb{P}(|X-\mathbb{E}[X]|>t)\leq \exp\left(-\frac{2t^2}{(t+\mathbb{E}[X])}\right).
    \]
\end{lemma}

We also employ the well-known Borel-Cantelli Lemma (see \cite{loeve1977elementary}).
\begin{lemma}[Borel-Cantelli Lemma]\label{lem:bor}
If $\{E_n\subseteq \Omega\}_{n\in \N}$ is a sequence of events such that
\[
\sum_{n\in \N} \mathbb{P}(E_n)<+\infty
\]
then
\[
\mathbb{P}\left(\limsup_{n\rightarrow +\infty} E_n\right)= \mathbb{P}\left(\bigcap_{n=1}^{+\infty}\bigcup_{k\geq n} E_k \right)=0.
\]
\end{lemma}
\begin{remark}[Chebishev inequality]\label{rmk:cheb}
Given $X:\Omega \rightarrow \R$ a random variable, a useful way to estimate the probability of the event $A:=\{\omega\in \Omega \ : \ X(\omega)\geq t\}$ is the so-called Chebishev inequality
    \[
    \mathbb{P}(A)=\int_A \d \mathbb{P}(\omega)\leq \frac{1}{t}\int_A X(\omega)\d \mathbb{P}(\omega)\leq \frac{1}{t} \int_\Omega X(\omega)\d \mathbb{P}(\omega)= \frac{\mathbb{E}[X]}{t}.
    \]
\end{remark}
Let $\{X_J \ : \ J\in \mathcal{I}_s(\R^d)\}$  be a countable family of independent Bernoulli random variables taking value $X_J=1$ with  probability $p$. The probability $\mathbb{P}_p$ is defined 
\[
\mathbb{P}_p(X_{J_1}=1,\ldots, X_{J_k}=1):=p^k
\]
and then extended to the $\sigma$-algebra generated by the cylinders
\[
A(J_1,\ldots,J_k)=\{X_{J_1}=1,\ldots, X_{J_k}=1\}.
\]
\begin{definition}\label{def:pathconn} 
We say that a subset of indexes $\mathcal{J}\subset \mathcal{I}_s(\R^d)$ is path connected if for any $J,J'$ there exists a family (called \textit{a path in $\mathcal{J}$}), $\{J_1,\ldots,J_k\}$ with 
\begin{itemize}
    \item[a)] $J_1=J, J_k=J'$;
    \item[b)] $J_i\in \mathcal{J}$ for all $i=1,\ldots k$;
    \item[c)] $|J_i-J_{i+1}|=s$ for all $i=1,\ldots, k-1$.
\end{itemize}
For $ \mathcal{J} \subset \mathcal{I}_s(Q)$, $J\in \mathcal{J}$ the set 
    \[
    U(J):=\{J'\in \mathcal{J} \ : \ \text{$J'$ is connected to $J$ through a path in $\mathcal{J}$}\}
    \]
is called \textit{the connected component of $\mathcal{J}\subset \mathcal{I}_s(\R^d)$ containing $J$}. Clearly if $J'\in U(J)$ then $U(J')=U(J)$ (see Figure \ref{fig:csi}).
\end{definition}

We will also employ the following Bernoulli  bond  percolation result. Lemma \ref{lem:perc} summarizes a combination of standard results in Bernoulli bond percolation theory (see  \cite{aizenman1987sharpness}, \cite{aizenman1984tree}, \cite{bollobas2006percolation},  \cite{grimmett2013percolation}, \cite{menshikov1986coincidence}). We underline that the choice of the origin $J=0$ in the following statement plays no role. In particular \eqref{eqn:oneCon},\eqref{eqn:twoCon} remain valid for a generic $J$.

\begin{lemma}\label{lem:perc}
Let $\{X_J \ : \ J\in \mathcal{I}_s(\R^d)\}$ be a countable family of independent Bernoulli random variables taking value $X_J=1$ with  probability $p$. Then if $p<p_c(d)$, being $p_c$  a critical value depending on the dimension only, the set $\{J\in \mathcal{I}_s(\R^d) : X_J=0\}$ is almost surely path connected. Moreover, there exists $\psi_p$ depending on $p$ only such that, setting $U(0)$ to be the (possibly empty) connected component of $\{J\in \mathcal{I}_s(\R^d) : X_J=1\}$ containing $J=0$, it holds
 \begin{equation}\label{eqn:oneCon}
    \mathbb{P}_{p}(\#(  U(0) )  > N)\leq e^{-\psi_p N } \ \ \ \text{for all $N\geq1$}.
\end{equation}
Finally for all $p<p'$.
\begin{equation}\label{eqn:twoCon}
 \mathbb{P}_{p}(\#(  U(0) )  > N)\leq  \mathbb{P}_{p'}(\#(  U(0) )> N).
\end{equation}
\end{lemma}
\rc{We now prove a statistical estimate which allows us to obtain information on those regions of the point cloud having the wrong number of interactions.}
\begin{lemma}\label{lem:DimConn}
Suppose that 
   \[
   \limsup_{\e\rightarrow +\infty} \frac{\e}{s_{\e}}=0.
   \]
Then there exists a subsequence $\{\e_n\}_{n\in\N}$ and a constant $\Lambda$ such that almost surely the following hold. For all $\kappa\in(0,1)$ we can find $n_0$ (depending on $\kappa$ and the realization only) for which (setting for shortness $s_{n}=s_{\e_{n}}$), $\mathcal{I}_{s_n}(Q)\setminus \mathcal{I}_{s_n}^{\kappa}(Q)$ is path connected for any $n\geq n_0$ and it holds
\begin{equation}\label{eqn:bndonconn}
    \sup\left\{\#(U) \ : \ \text{$U$ conn. comp. of $\mathcal{I}_{s_n}^{\kappa}(Q)$}\right\}\leq \Lambda  \log(s_n^{-1}) \ \  \text{for all $n\geq n_0$}.
\end{equation}
\end{lemma}
\begin{proof}
For any $J\in \mathcal{I}_{s_{\e}}^{\kappa}(Q)$ consider $U(J)$ to be the (possibly empty) connected component of  $\mathcal{I}_{s_{\e}}^{\kappa}(Q)$ containing $J$. Setting, for $J\in \mathcal{I}_{s_{\e}}(Q)$,    \begin{equation}
    X_J^\e :=\left\{\begin{array}{lr}
        1 & \text{if $J\in \mathcal{I}_{s_\e}^{\kappa}(Q)$} \\
        0 & \text{otherwise}.
    \end{array} \right.    
    \end{equation}
we have, by invoking Lemma \ref{lemma:CCconcc}, that 
\[
p_{\e}(\kappa)=\mathbb{P}(X_J^{\e}=1)\leq e^{-\frac{2\kappa^2}{1+\kappa} \frac{s_{\e}^d}{\e^d}}.
\]
 In particular, for any $\kappa>0$, $p_{\e}(\kappa)\rightarrow 0$. Thus, for any $\kappa$ we can find $\e_0=\e_0(\kappa,d)$ such that $p_{\e}(\kappa)\leq p_0<p_c(d)$ and $\mathcal{I}_{s_\e}(Q)\setminus \mathcal{I}_{s_\e}^{\kappa}(Q)$ is almost surely path connected for all $\e\leq \e_0$. In particular
    \[
    \mathbb{P}_{p_{\e}(\kappa)}(\#(U(0) )\geq N)\leq    \mathbb{P}_{p_{0}}(\#(U(0) )\geq N)\leq e^{-\psi_{p_0}N}.
    \]
for some $\psi_{p_0}$. Fix $\Lambda'>d$ and consider the subsequence $\{\e_{n}\}_{n\in\N}$ such that
\[
s_{\e_{n}}^{\Lambda'-d}\leq \frac{1}{n^2}.
\]
Setting $s_n:=s_{\e_{n}}$, let us then consider the events
    \[
    A_{n}:=\left\{\omega \in \Omega  \ : \ \text{exists $J\in \mathcal{I}_{s_n}(Q)$} \ \text{s.t.} \ \#(U(J))\geq \frac{\Lambda'}{\psi_{p_0}} \log(s_{n}^{-1}) \right\}.
    \]
Then  
    \begin{align*}
        \mathbb{P}(A_n)&\leq \mathbb{P}\left(\bigcup_{J\in \mathcal{I}_{s_n}(Q)} \left\{\omega \in \Omega \ : \ \#(U(J))\geq \frac{\Lambda'}{\psi_{p_0}} \log(s_{n}^{-1})\right\} \right)\\
        &\leq \sum_{J\in \mathcal{I}_{s_n}(Q)}  \mathbb{P}\left(  \left\{\omega \in \Omega \ : \ \#(U(J))\geq \frac{\Lambda'}{\psi_{p_0}} \log(s_{n}^{-1})\right\} \right)\\
        &=\sum_{J\in \mathcal{I}_{s_n}(Q)}    \mathbb{P}_{p_{\e_n}(\kappa)}\left(\#(U(J))\geq\frac{\Lambda'}{\psi_{p_0}} \log(s_{n}^{-1})\right)\\
        &\leq s_n^{\Lambda'} \#(\mathcal{I}_{s_n}(Q))\leq C s_n^{\Lambda'-d}\leq C n^{-2}.
    \end{align*} 
By invoking Borel-Cantelli Lemma \ref{lem:bor} we conclude that the probability that $A_n$ occurs infinitely many time is $0$. Thus, naming $\Lambda:=\frac{\Lambda'}{\psi_{p_0}} $ we conclude.
\end{proof}

We will also make use of the following Proposition, providing bounds on the cardinality of $\mathcal{I}_{s_{\e}}^{\kappa}(Q)$ (up to select a subsequence $\{\e_n\}_{n\in \N}$).
\begin{proposition}\label{prop:decay}
Let $s_{\e}$ be a sequence satisfying
\begin{equation}\label{eqn:regime2}
    \liminf_{\e \rightarrow 0}\frac{s_{\e}}{\e\log(\e^{-d})^{\sfrac{1}{d}}}=\beta>0.
\end{equation}
Then the following hold: 
\begin{itemize}
    \item [a)] If $\beta>\frac{1}{\gamma^{1/d}}$, there exists $0<\kappa_0<1$ and a subsequence $\{\e_n\}_{n\in \N}$ such that, setting $s_n=s_{\e_n}$, for all $\kappa_0 \geq \kappa<1$ and almost surely we can find $n_0$ for which
    \[
    \#(\mathcal{I}_{s_{_n}}^{\kappa}(Q))=0 \ \ \ \text{for all $n\geq n_0$};
    \]
    \item[b)] If $0<\beta\leq \frac{1}{\gamma^{1/d}}$ there exist two constants $\kappa_0>0$, $\varrho_0>0$ depending on $\beta,d,\gamma$ only and a subsequence $\{\e_n\}_{n\in \N}$ such that, setting $s_n:=s_{\e_n}$, for all $\kappa_0\leq \kappa< 1$ and almost surely we can find $n_0$ for which 
    \[
    \#(\mathcal{I}_{s_n}^{\kappa}(Q))\leq \frac{\e_n^{\varrho_0 }}{s_n^d}  \ \ \text{for all $n\geq n_0$.}
    \]
\end{itemize}
\end{proposition}
\begin{proof}
Set $\sigma_{\e}:=\e \log(\e^{-d})^{\sfrac{1}{d}}$. We can select a subsequence such that, setting $\beta_n=\beta-\frac{1}{n}$, $\sigma_n:=\sigma_{\e_n}$, we have
        \[
        s_{n}\geq \beta_n \sigma_{n}.
        \]
Notice that (cf. Remark \ref{rmk:exp})
    \[
   \mathbb{E}[\eta_{\e}(Q_{s}(J))]=\gamma\frac{s^d}{\e^d}.
    \]
In particular, Lemma \ref{lemma:CCconcc} yields
    \begin{align*}
        p_{n}&:= \mathbb{P}\left(\left|\eta_{\e_n}(Q_{s_{n}}(J))-\mathbb{E}[\eta_{\e_n}(Q_{s_n}(J))]\right|>\kappa\gamma\frac{s_{n}^d}{\e_n^d}\right)\leq \exp\left(-\frac{2\kappa^2\gamma}{(1+\kappa)}\frac{s_{n}^d}{\e_n^d}\right)\\
        &\leq \exp\left(-\frac{2\kappa^2\gamma}{(1+\kappa)}\log\e_n^{-d \beta_n^d}\right) = \e_n^{\beta_n^d d  \gamma \frac{2\kappa^2}{(1+\kappa)}}.
    \end{align*}
Set,  for $J\in\mathcal{I}_{s_n}(Q) $ 
    \begin{equation}
    X_J^{\kappa,n} :=\left\{\begin{array}{lr}
        1 & \text{if $J\in \mathcal{I}_{s_n}^{\kappa}(Q)$} \\
        0 & \text{otherwise}.
    \end{array} \right.    
    \end{equation}
Notice that each $X_J^{\kappa,n}$ is a Bernoulli random variable attaining value $1$ with probability $p_n$ and $0$ with probability $1-p_n$ (Here we are omitting the dependence of $p_n$ on $\kappa$ to lighten the notation). Setting then
\[
T_n^{\kappa} :=\sum_{J\in \mathcal{I}_{s_n}(Q)}X_J^{\kappa,n}=\#(\mathcal{I}_{s_n}^{\kappa}(Q))
\]
we have that
    \[
    \mathbb{E}[T_n^{\kappa}]=\#( \mathcal{I}_{s_n}(Q)) p_n\leq C \frac{\e_n^{\beta_n^d d  \gamma \frac{2\kappa^2}{(1+\kappa)}}}{s_n^d}
    \]
for a constant $C$ depending on the dimension only (that in the sequel may vary from line to line). 
 Setting 
    \[
   \varrho_n(\kappa):= \beta_n^d d   \gamma \frac{2\kappa^2}{(1+\kappa)} 
    \]
 the previous estimate writes 
\[
 \mathbb{E}[T_n^{\kappa}]\leq  C \frac{ \e_n^{ \varrho_n(\kappa)} }{s_n^d}.
 \]
 
\textbf{Proof of assertion a).} If $\beta^d> \frac{1}{\gamma} $ then, for some $n_0$, $\beta_n^d> \frac{1}{\gamma}$ for $n\geq n_0$. Then, for some $0 < \kappa_0 < 1$ we have $\varrho_n(\kappa_0)>d$ whenever  $n\geq n_0$. Also recalling that, from the very definition of $\sigma_{\e}$, it holds that $\left(\frac{s_n}{\e_n}\right)^d\geq \frac{1}{d} \frac{\beta_n^d}{\log(\e_n^{-1})}$ we have
 
\[
    \mathbb{E}[T_n^{\kappa_0}] \leq C \frac{\e_n^{\varrho_n(\kappa_0)-d}}{\log(\e_n^{-1})}.
    \]
Consider the subsequence $\{\e_{n_j}\}_{j\in \N}$ such that
\[
\e_{n_j}^{(\varrho_j(\kappa_0)-d)/2}\leq \frac{1}{j^2}
\]
where we abbreviate $ \varrho_j:=\varrho_{n_j} $ for the sake of clarity. Define the events
\[
A_j(\kappa_0):=\left\{\omega\in \Omega \ :\ T_{n_j}^{\kappa_0}(\omega) \geq\frac{ \e_{n_j}^{(\varrho_j(\kappa_0)-d)/2}}{\log(\e_{n_j}^{-1})}\right\}.
\]
Then Chebishev inequality (cf. Remark \ref{rmk:cheb}) yields
    \begin{align*}
        \mathbb{P}(A_j(\kappa_0))\leq \frac{\mathbb{E}[T_{n_j}^{\kappa_0}] \log(\e_{n_j}^{-1})}{\e_{n_j}^{(\varrho_j(\kappa_0)-d)/2}}\leq C\e_{n_j}^{(\varrho_j(\kappa_0)-d)/2}\leq C \frac{1}{j^2}.
    \end{align*}
    In particular
        \[
        \sum_{j\in \N} \mathbb{P}(A_j(\kappa_0))<+\infty.
        \]
    Borel-Cantelli Lemma \ref{lem:bor} now implies that
        \[
        \mathbb{P}\left(\bigcap_{j\in \N}\bigcup_{m\geq j} A_m(\kappa_0)\right)=0.
        \]
Set $\Omega':=\Omega \setminus \bigcap_{j\in \N}\bigcup_{m\geq j} A_m(\kappa_0)$. We now observe that if $\omega \in \Omega'$ then there exists $j_0$ such that 
\[
 \#(\mathcal{I}_{s_{n_j}}^{\kappa_0}(Q))< \frac{ \e_{n_j}^{(\varrho_j(\kappa_0)-d)/2}}{\log(\e_{n_j}^{-1})}\ \ \ \text{for all $j\geq j_0$}.
\]
Since $\varrho_j(\kappa_0)-d>0$ we conclude that 
\[
 \#(\mathcal{I}_{s_{n_j}}^{\kappa_0}(Q)) =0 \ \ \text{for all $j$ big enough}.
\]
Since, for $\kappa>\kappa_0$ we have $\mathcal{I}_{s_{n_j}}^{\kappa}(Q)\subset \mathcal{I}_{s_{n_j}}^{\kappa_0}(Q) $ we conclude the proof of assertion a). \\

\textbf{Proof of assertion b).} If $0<\beta^d\leq \frac{1}{\gamma}$ we have (for all $n$ big enough)
\[
d>\varrho_n(\kappa)>\left(\frac{\beta}{2}\right)^d d   \gamma \frac{2\kappa^2}{(1+\kappa)}=\varrho(\kappa)>0 
\]
for all $\kappa\in(0,1)$ and then 
\[
    \mathbb{E}[T_n^{\kappa}]\leq C \frac{\e_n^{\varrho(\kappa)}}{s_n^d}.
    \]
Fix $\kappa_0\in (0,1)$ and consider the subsequence $\{\e_{n_j}\}_{j\in \N}$ such that
\[
\e_{n_j}^{\varrho(\kappa_0)/2}\leq \frac{1}{j^2}.
\]
We define the events
\[
A_j(\kappa_0):=\left\{\omega\in \Omega \ :\ T_{n_j}^{\kappa_0}(\omega) \geq\frac{ \e_{n_j}^{\varrho(\kappa_0)/2}}{s_j^d}\right\}.
\]
Then Chebishev inequality (cf. Remark \ref{rmk:cheb}) again yields
    \begin{align*}
        \mathbb{P}(A_j(\kappa_0))\leq \frac{\mathbb{E}[T_{n_j}^{\kappa_0}] s_j^d}{\e_{n_j}^{\varrho(\kappa_0)/2}}\leq C\e_{n_j}^{\varrho(\kappa_0)/2}\leq \frac{1}{j^2}.
    \end{align*}
    In particular
        \[
        \sum_{j\in \N} \mathbb{P}(A_j(\kappa_0))<+\infty.
        \]
By invoking again Borel-Cantelli Lemma \ref{lem:bor} we achieve
        \[
        \mathbb{P}\left(\bigcap_{j\in \N}\bigcup_{m\geq j} A_m(\kappa_0)\right)=0.
        \]
Set $\Omega'':=\Omega \setminus \bigcap_{j\in \N}\bigcup_{m\geq j} A_m(\kappa_0)$. We now observe that if $\omega \in \Omega''$ then there exists $j_0$ such that 
\[
 \#(\mathcal{I}_{s_{n_j}}^{\kappa_0,\omega}(Q))\leq \frac{ \e_{n_j}^{\varrho(\kappa_0)/2}}{s_n^d}\ \ \ \text{for all $j\geq j_0$}.
\]
Since, for $\kappa>\kappa_0$, we have $\mathcal{I}_{s_{n_j}}^{\kappa}(Q)\subset \mathcal{I}_{s_{n_j}}^{\kappa_0}(Q) $ we conclude by setting $\varrho_0:=\varrho(\kappa_0)/2$.
\end{proof}

\section{Proof of Compactness Theorem \ref{MainThm:CMP}}\label{sct:proof}

\begin{figure}
    \centering
    \includegraphics[scale=0.6]{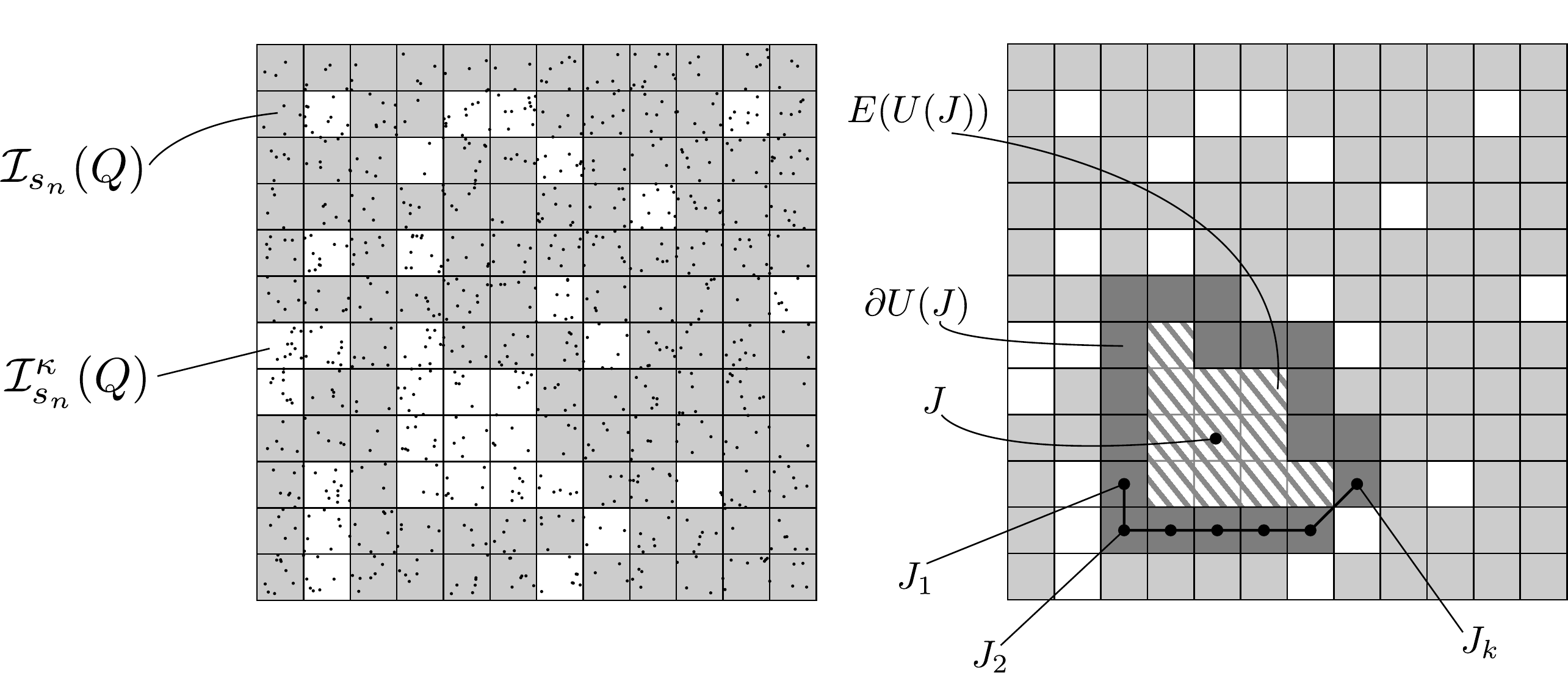}
    \caption{On the left: The percolation properties of $\mathcal{I}_{s_n}^{\kappa}(Q)$ stated in  Lemma \ref{lem:DimConn} and Proposition \ref{prop:decay}. On the right: a depiction of a diagonal path $\mathbf{t}_{\partial U}$, a connected component $U=U(J)$ of $\mathcal{I}_{s_n}^{\kappa}(Q)$ and the boundary $\partial U(J)$. }
    \label{fig:csi}
\end{figure}
Let us introduce some notation needed in the proof of the extension Lemma \ref{lem:extension}. Given $U\subset \mathcal{J}$ a connected component of a subfamily of indexes $\mathcal{J}\subset \mathcal{I}_{s_\e}(Q)$ we denote 
\[
\rc{\partial U:=\left\{J\in \mathcal{I}_{s_\e}(Q)\setminus \mathcal{J}\ | \ \exists  \ J'\in  U  \ : \partial Q_{s_{\e}}(J')\cap \partial Q_{s_{\e}}(J)\neq \emptyset \right\}.}
\]
Notice that we allow also paths that can walk in diagonal, since our energy allows us to control also the diagonal direction. In particular, for $J\in \mathcal{I}_{s_\e}(Q)$ we define the set of \textit{neighbors} of $J$ as
    \[
    N_{\e}(J):=\{J'\in (Q_{3s_\e}(J)\setminus\{J\})\cap \mathcal{I}_{s_\e}(Q)\}.
    \]
\begin{definition}[diagonal paths]
We say that $\mathbf{t}(J,J')=\{J_1,\ldots,J_k\}\subset \mathcal{I}_{s_\e}(Q)$ is a \textit{diagonal path} connecting $J$ to $J'$ if $J_1=J$, $J_k=J'$ and $J_{i+1} \in N_\e(J_i)$ for all $i=1,\ldots, k-1$.
\end{definition}
Motivated by Lemma \ref{lem:DimConn} and Proposition \ref{prop:decay}, we introduce the notion of \textit{controlled subfamily of indexes $\{\mathcal{J}_n \subset \mathcal{I}_{s_n}(Q)\}_{n\in \N}$}.
\begin{definition}\label{def:contrsub}
Let $\{s_n\}_{n\in \N}$ be a vanishing sequence. We say that $\{\mathcal{J}_n \subset \mathcal{I}_{s_n}(Q)\}_{n\in \N}$ is a \textit{controlled subfamily of indexes} if there exists $n_0\in \N$, $\varrho_0\in (0,1)$, $\Lambda\in \R_+$ universal constants such that the following are in force for all $n\geq n_0$:
\begin{itemize}
    \item[a)] $\mathcal{I}_{s_n}(Q)\setminus \mathcal{J}_n$ is path connected in the sense of Definition \ref{def:pathconn};
      \item[b)] there exists $0<\varrho_0 <1$ such that $\#(\mathcal{J}_n)s_n^d \leq s_n^{\varrho_0}$;
     \item[c)] for each connected component $U\subset \mathcal{J}_n$ it holds $ \#(U)\leq \Lambda \log(s_n^{-1})$.
\end{itemize}
\end{definition}
\begin{remark}
Up to extract a subsequence, by combining Lemma \ref{lem:DimConn} and Proposition \ref{prop:decay} we can ensure that (for a suitable $\kappa$) the family $\mathcal{J}_n=\mathcal{I}_{s_n}^\kappa(Q)$ is a \textit{controlled subfamily of index} (we refer the reader to Figure \ref{fig:csi}).
\end{remark}
\begin{lemma}\label{lem:controlledpaths}
Let $\{\mathcal{J}_n\}_{n\in \N}$ be a controlled subfamily of indexes. There exists a universal constant $C$ such that for any $U\subset \mathcal{J}_n$ connected component of $\mathcal{J}_n$, and for any $J,J'\in \partial U$, we can find a diagonal path $\mathbf{t}_{\partial U}(J,J')\subset \partial U$ such that
\begin{equation}\label{eqn:lengthpath}
\#(\mathbf{t}_{\partial U}(J,J'))\leq C \#(\partial U).
\end{equation}
\end{lemma}
\begin{proof}
Observe that the set 
    \[
    E(U):=\bigcup_{J\in U} Q_{s_n}(J)
    \]
is connected and simply connected due to properties a), b), c) of Definition \ref{def:contrsub}. In particular the set $\partial E(U)$ is connected and thus connected by arc. Fix $J,J'\in \partial U$ and pick 
    \[
    x_J\in \partial Q_{s_n}(J)\cap \partial E(U), \ \ \  y_{J'}\in \partial Q_{s_n}(J')\cap \partial E(U).
    \]
Consider then an arc $\tau:[0,1]\rightarrow \partial E(U)$ such that $\tau(0)=x_J$, $\tau(1)=y_J'$. Since $\tau\in \partial E(U)$ we can find $J_1,\ldots,J_k \in \partial U$, $t_1,\ldots, t_k\in (0,1)$ such that $J_1=J$, $J_k=J'$ and $\tau(t)\in Q_{s_n}(J_i)$ for $t\in(t_i,t_{i+1})$. By eventually reducing loops and repetition, without loss of generality we can assume that $\tau$ is a simple arc and do not passes through each $J_i$ more than once. Now by defining $\mathbf{t}_{\partial U}(J,J'):=\{J_1,\ldots,J_k\}$, by construction we have that $\mathbf{t}_{\partial U}$ is a diagonal path. The bound \eqref{eqn:lengthpath} is also immediate. 
\end{proof}
\begin{remark}\label{rmk:connect}
Notice that, if property a) of Definition \ref{def:contrsub} fails then Lemma \ref{lem:controlledpaths} is, in general, false.
\end{remark}

\begin{lemma}[An extension Lemma]\label{lem:extension}
Let $\{s_n\}_{n\in\N}$ be a vanishing sequence. Let $\{\mathcal{J}_n\subset \mathcal{I}_{s_n}(Q)\}_{n\in\N}$ be a controlled subfamily of indexes in the sense of Definition \ref{def:contrsub}. Then, for any sequence  of functions  $\{u_{n}: \mathcal{I}_{s_{n}}(Q)\setminus \mathcal{J}_{n} \rightarrow \R^+\}_{n\in\N}$ such that
    \begin{equation}\label{eqn:bndLimit}
    \sup_{n\in \N}\left\{\sum_{\substack{J,J' \in \mathcal{I}_{s_{n}}(Q)\setminus \mathcal{J}_{n} \\ J' \in N_{\e_n}(J)}}|u_{n}(J)-u_{n}(J')|^p s_{n}^{d-p} + \sum_{\substack{J\in\mathcal{I}_{s_{n}}(Q)\setminus \mathcal{J}_n}} |u_n(J)|^ps_{n}^{d}  \right\}<+\infty
    \end{equation}
 we can find a sequence of extended functions $\{Tu_n: \mathcal{I}_{s_n}(Q) \rightarrow \R^+\}_{n\in\N}$ such that
\begin{align*}
    Tu_n&=u_n\ \ \ \text{on $\mathcal{I}_{s_n}(Q)\setminus \mathcal{J}_{n} $ }
\end{align*}
and for all $q<p$
  \begin{equation}\label{eqn:bndLimit2}
        \sup_{n\in \N} \left\{\sum_{\substack{J,J'\in \mathcal{I}_{s_n}(Q) \\ |J-J'|=s_n}} |Tu_n(J)-Tu_n(J')|^q s_n^{d-q} + \sum_{\substack{J\in\mathcal{I}_{s_n}(Q)}} |T u_n(J)|^q s_n^{d}  \right\}  <+\infty.
    \end{equation}
\end{lemma}
\begin{remark}
Notice that we assume a bound \eqref{eqn:bndLimit} involving interaction also in diagonal directions and we conclude a control on the energy of $Tu_n$ \eqref{eqn:bndLimit2} in terms of nearest neighbors. Bound \eqref{eqn:bndLimit2} could, in principle, be improved to a bound involving the full family of neighbors $N_{\e_n}(J)$ (consistently with \eqref{eqn:bndLimit}) by carrying a careful analysis of the diagonal interactions of $Tu_n$'s. However, as clarified by Lemma \ref{lem:alici}, bound \eqref{eqn:bndLimit2} is already strong enough to conclude  compactness for the sequence $Tu_n$.
\end{remark}
\begin{proof}
We set, for $U$ a connected component of $\mathcal{J}_n$:
\begin{align*}
    (u_n)_U&:=\frac{1}{\#( \partial U)}\sum_{J\in  \partial U} u_n(J).
\end{align*}
We then set
    \begin{equation}
    Tu_n(J)=\left\{ \begin{array}{ll}
     u_n(J) & \ \ \text{if $J\in\mathcal{I}_{s_n}(Q)\setminus \mathcal{J}_{n}$},\\
    (u_n)_U \ & \ \text{if $J\in U$, for some $U$ connected component of  $ \mathcal{J}_n$}.
    \end{array}\right.
    \end{equation}
We now estimate the energy of $Tu_n$.
\begin{align}
  \sum_{\substack{J,J'\in \mathcal{I}_{s_n}(Q) \\ |J-J'|=s_n}} |Tu_{n}(J)-Tu_{n}(J')|^q s_n^{d-q}  \leq & C  \sum_{\substack{J,J'\in \mathcal{I}_{s_n}(Q)\setminus \mathcal{J}_n \\ |J-J'|=s_n}}  |u_n(J)-u_n(J')|^q s_n^{d-q} \nonumber \\ 
    &+C\sum_{\substack{U\subset \mathcal{J}_n \\ \text{connected}\\\text{component}} }\sum_{\substack{J\in\partial  U}} |u_n(J)-(u_n)_U|^q s_n^{d-q}. \label{eqn:estimateEnr}
\end{align}
Let us treat the second sum separately. Fix $U\subset \mathcal{J}_n $ a connected component and observe, by considering the paths given by Lemma \ref{lem:controlledpaths}, that
\begin{align*}
    \sum_{\substack{J\in\partial  U}}& |u_n(J)-(u_n)_U|^q s_n^{d-q} \leq \frac{1}{\#(\partial  U)}  \sum_{\substack{J,J'\in\partial  U}} |u_n(J)-u_n(J')|^q s_n^{d-q}\\
  &\leq \frac{1}{\#(\partial U)}  \sum_{\substack{J,J'\in\partial U}} (\#(\mathbf{t}_{\partial U}(J,J'))^{q-1}\sum_{I\in \mathbf{t}_{\partial U}(J,J')}\sum_{\substack{ I' \in \mathcal{I}_{s_n}(Q)\setminus \mathcal{J}_n :\\ I'\in N_{\e_n}(I)} } |u_n(I)-u_n(I')|^q s_n^{d-q} \\
    &\leq \frac{C}{\#(\partial U)}  \sum_{\substack{J,J'\in\partial U}} (\#(\mathbf{t}_{\partial U}(J,J'))^{q-1} \sum_{I\in \mathbf{t}_{\partial U}(J,J')}\left(\sum_{\substack{I'\in \mathcal{I}_{s_n}(Q)\setminus \mathcal{J}_n :\\ I'\in N_{\e_n}(I)} } |u_n(I)-u_n(I')|^p s_n^{\frac{dp}{q}-p}\right)^{q/p} \\
        &\leq \frac{C s_n^{\frac{d(p-q)}{p}}}{\#(\partial U)}  \sum_{\substack{J,J'\in\partial U}}  (\#(\mathbf{t}_{\partial U}(J,J'))^{q-1}\sum_{I \in \mathbf{t}_{\partial U}(J,J')}\left(\sum_{\substack{I'\in \mathcal{I}_{s_n}(Q)\setminus \mathcal{J}_n :\\ I'\in N_{\e_n}(I)} } |u_n(I)-u_n(I')|^p s_n^{d-p}\right)^{q/p}
\end{align*}
where we applied Jensen's inequality with the function $x\mapsto x^{p/q}$. Call
    \[
   Eu_n(I):= \sum_{\substack{I'\in \mathcal{I}_{s_n}(Q)\setminus \mathcal{J}_n :\\ I'\in N_{\e_n}(I)} } |u_n(I)-u_n(I')|^p s_n^{d-p}.
    \]
Then, an application of Holder's inequality yields
\begin{align*}
     \sum_{\substack{J\in\partial U}} & |u_n(J)-(u_n)_U|^q s_n^{d-q} \leq  \frac{C s_n^{\frac{d(p-q)}{p}}}{\#(\partial  U )}  \sum_{\substack{J,J'\in\partial U}}  (\#(\mathbf{t}_{\partial U}(J,J'))^{q-1}\sum_{I \in \mathbf{t}_{\partial U}(J,J')}\left(Eu_n(I)\right)^{q/p}\\
     &\leq  C \frac{s_n^{\frac{d(p-q)}{p}}}{\#(\partial  U)}  \sum_{\substack{J,J'\in\partial  U}}  (\#(\mathbf{t}_{\partial U}(J,J'))^{q-1+\frac{p-q}{p}}\left(\sum_{I\in \mathbf{t}_{\partial U}(J,J')} Eu_n(I)\right)^{q/p}\\
          &\leq C s_n^{\frac{d(p-q)}{p}}   \#(\partial  U)^{q+\frac{p-q}{p}}\left(\sum_{I\in \partial  U} Eu_n(I)\right)^{q/p} .
\end{align*}
In the last inequality we applied \eqref{eqn:lengthpath}. In particular, by applying again Holder inequality
\begin{align*}
 \sum_{\substack{U\subset \mathcal{J}_n \\ \text{connected}\\\text{component}}  } &\sum_{\substack{J\in\partial  U}} |u_n(J)-(u_n)_U|^q s_n^{d-q}\\
& \leq s_n^{\frac{d(p-q)}{p}}  \sum_{\substack{U\subset \mathcal{J}_n \\ \text{connected}\\\text{component}}  }  \#(\partial U)^{q+\frac{p-q}{p}}\left(\sum_{I\in \partial U} Eu_n(I)\right)^{q/p}\\
& \leq s_n^{\frac{d(p-q)}{p}} \left( \sum_{\substack{U\subset \mathcal{J}_n \\ \text{connected}\\\text{component}}  } \#(\partial U)^{\frac{pq}{p-q}+1}\right)^{\frac{p-q}{p}}\left(  \sum_{\substack{U\subset \mathcal{J}_n \\ \text{connected}\\\text{component}}  } \sum_{I\in \partial  U } Eu_n(I)\right)^{q/p}.
\end{align*}
By assumption \eqref{eqn:bndLimit} we have 
\[
 \left(   \sum_{\substack{U\subset \mathcal{J}_n \\ \text{connected}\\\text{component}}  } \sum_{I\in \partial  U } Eu_n(I)\right)^{q/p}<M.
\]
Moreover, by Properties b), c) of $\mathcal{J}_n$ (Definition \ref{def:contrsub}) we can infer (since $\#(\partial U)\leq C \#(U)$ for some universal constant $C$)
\begin{align*}
    s_n^{\frac{d(p-q)}{p}} \left( \sum_{\substack{U\subset \mathcal{J}_n \\ \text{connected}\\\text{component}}  }  \#(\partial U)^{\frac{pq}{p-q}+1}\right)^{\frac{p-q}{p}}&\leq C\left(  \log(s_n^{-1})^{\frac{pq}{p-q}} \sum_{\substack{U\subset \mathcal{J}_n \\ \text{connected}\\\text{component}}  } \#(\partial  U)s_n^{d}\right)^{\frac{p-q}{p}}\\
    &\leq C\left(  \log(s_n^{-1})^{\frac{pq}{p-q}}\#(\mathcal{J}_n) s_n^{d}\right)^{\frac{p-q}{p}}\\
     &\leq C\left(  \log(s_n^{-1})^{\frac{pq}{p-q}}s_n^{\varrho_0}\right)^{\frac{p-q}{p}}.
\end{align*}
This is enough to deduce that, for $q<p$,
\begin{equation} 
    \lim_{n\rightarrow +\infty}\sum_{\substack{U\subset \mathcal{J}_n \\ \text{connected}\\\text{component}}  } \sum_{\substack{J\in\partial  U}} |u_n(J)-(u_n)_U|^q s_n^{d-q}=0
\end{equation}
which, combined with \eqref{eqn:estimateEnr}, \eqref{eqn:bndCMP} is enough to conclude a uniform bound
    \begin{equation}\label{uniene}
         \sup_{n\in \N} \left\{\sum_{\substack{J,J'\in \mathcal{I}_{s_n}(Q) \\ |J-J'|=s_n}} |Tu_n(J)-Tu_n(J')|^q s_n^{d-q}\right\} <+\infty.
    \end{equation}
    The $L^q$ norm part can be estimated in a similar way as follows.
    \begin{align*}
        \sum_{J\in \mathcal{I}_{s_n}(Q)} |Tu_n(J)|^{q} s_n^d &\leq \sum_{J\in \mathcal{I}_{s_n}(Q)\setminus \mathcal{J}_n} |u_n(J)|^{q} s_n^d+\sum_{\substack{U\subset \mathcal{J}_n \\ \text{connected}\\\text{component}}  } \frac{\#(U)}{\#(\partial U)} \sum_{J\in \partial U}|u_n(J)|^q s_n^d.
    \end{align*}
 Observe that, recalling the notation $E(U):=\bigcup_{J\in U} Q_{s_{n}}(J)$, we have (for some constant $C$ depending on the dimension only) 
\begin{equation*}
\begin{split}
   C^{-1} \#(\partial U) s_n^{d-1}\leq P(&E(U))\leq C  \#(\partial U) s_n^{d-1}\\
    C^{-1} \#(\partial U) s_n^{d}\leq |&E(U)| \leq C  \#(\partial U) s_n^{d}
    \end{split}
\end{equation*}
(being $P(\cdot)$ the distributional perimeter). Thus, by the isoperimetric inequality (see for instance \cite{maggi})
 
    \[
    \frac{\#(U)}{\#(\partial U)}\leq C \frac{1}{s_n}\frac{|E(U)|}{P(E(U))}\leq C \frac{1}{s_n} P(E(U))^{1/(d-1)}\leq C \#(\partial U)^{1/(d-1)}
    \]
where $C$ is a universal constant depending only on the dimension. Thus
    \begin{align*}
\sum_{\substack{U\subset \mathcal{J}_n \\ \text{connected}\\\text{component}}  } \frac{\#(U)}{\#(\partial U)} \sum_{J\in \partial U}|u_n(J)|^q s_n^d &\leq C \sum_{\substack{U\subset \mathcal{J}_n \\ \text{connected}\\\text{component}}  } \#(\partial U)^{1/(d-1)} \sum_{J\in \partial U}|u_n(J)|^q s_n^d.
    \end{align*}
Also
    \begin{align*}
        \sum_{J\in \partial U}|u_n(J)|^q s_n ^d &\leq s_n^d \#(\partial U)^{(p-q)/p} \left(\sum_{J\in \partial U}|u_n(J)|^p \right)^{q/p} \\
        &\leq s_n^{d(p-q)/p} \#(\partial U)^{(p-q)/p} \left(\sum_{J\in \partial U}|u_n(J)|^p s_n^d \right)^{q/p} .
    \end{align*}
Call    
\[
\|u_n\|_{\ell^p(\partial U)}:=\left(\sum_{J\in \partial U}|u_n(J)|^p s_n^d\right)^{1/p}
\]
hence    
\begin{align*}
\sum_{\substack{U\subset \mathcal{J}_n \\ \text{connected}\\\text{component}}  } & \frac{\#(U)}{\#(\partial U)}   \sum_{J\in \partial U}|u_n(J)|^q s_n ^d \leq C s_n^{d(p- q)/p}\sum_{\substack{U\subset \mathcal{J}_n \\ \text{connected}\\\text{component}}  }  \#(\partial U)^{\frac{1}{d-1}+\frac{p-q}{p}}    \|u_n\|_{\ell^p(\partial  U)}^q\\
&\leq  \left( \sum_{\substack{U\subset \mathcal{J}_n \\ \text{connected}\\\text{component}}  }  \#(\partial U)^{\frac{p}{(p-q)(d-1)}+1} s_n^{d} \right)^{\frac{p-q}{p}}\left( \sum_{\substack{U\subset \mathcal{J}_n \\ \text{connected}\\\text{component}}  }  \|u_n\|_{\ell^p(\partial U)}^p\right)^{q/p}.
    \end{align*}
By assumption
\[
\left(\sum_{\substack{U\subset \mathcal{J}_n \\ \text{connected}\\\text{component}}  } \|u_n\|_{\ell^p(\partial U)}^p\right)^{q/p}<M.
\]
Moreover, by properties b), c) of $\mathcal{J}_n$ (Definition \ref{def:contrsub})  using again that $\#\partial U \leq  C \# U$ we have
\begin{align*}
    \left( \sum_{\substack{U\subset \mathcal{J}_n \\ \text{connected}\\\text{component}}  } \#(\partial  U)^{\frac{p}{(p-q)(d-1)}+1} s_n^{d} \right)^{\frac{p-q}{p}}\leq C \left(\log(s_n^{-1})^{\frac{p}{(p-q)(d-1)}} s_n^{\varrho_0}  \right)^{\frac{p-q}{p}}.
\end{align*}
We conclude that 
\begin{equation}\label{finiteNorm}
   \lim_{n\rightarrow +\infty} \sum_{J\in \mathcal{I}_{s_n}(Q)} |Tu_n(J)|^{q} s_n^d\leq C \sum_{J\in \mathcal{I}_{s_n}(Q)\setminus\mathcal{J}_n } |u_n(J)|^{p} s_n^d.
\end{equation}
In particular, by combining \eqref{finiteNorm}, \eqref{eqn:bndLimit} and \eqref{uniene} we conclude \eqref{eqn:bndLimit2}. 
\end{proof}
The last ingredient required to the proof of Theorem \ref{MainThm:CMP} is the following Lemma \ref{lem:alici}, which comes as a consequence of \cite[Theorem 3.1]{alicandro2004general}. \begin{lemma}\label{lem:alici}
Let $\{u_s: \mathcal{I}_s(Q)\rightarrow \R \}_{s\in \R_+}$ be a sequence of function such that
	\[
	\sup_{s\in \R_+}\left\{\sum_{\substack{ J,J' \in \mathcal{I}_s(Q)   :\\  |J-J'|=s  }} |u_s(J)-u_s(J')|^q s^{d-q}+\sum_{J\in \mathcal{I}_s(Q)} |u_s(J)|^q s^d \right\} <+\infty.
	\]
Then there exists a function $u\in W^{1,q}(Q)$ and a subsequence $\{s_l\}_{l\in \N}$ such that the piecewise constant extension  of $u_{s_l}$ :
\begin{align*}
    \hat{u}_{s_l}(x):=\sum_{J\in \mathcal{I}_{s_l}(Q)} u_{s_l}(J) \ca_{Q_{s_l}(J)}(x) 
\end{align*}
converges to $u$ in $L^q(Q)$.
\end{lemma} 
We are now ready to prove Theorem \ref{MainThm:CMP}.
\begin{proof}[Proof of Theorem \ref{MainThm:CMP}]
Define $s'_{\e}:=\frac{s_{\e}}{4\sqrt{d}}$. Notice that with this choice we have
    \begin{equation}\label{sqr}
    Q_{3 s_{\e}'}(J)\subseteq B_{s_\e}(x) \ \ \ \text{for all $x\in Q_{s_{\e'}}(J)$}.
    \end{equation}
Notice that $s_{\e}'$ still satisfies \eqref{eqn:regime2}.  By invoking Lemma \ref{lem:DimConn} and Proposition \ref{prop:decay} we can find $\kappa$ and extract a subsequence (still denoted by $\e_n, s'_n$) for which $\mathcal{J}_n:=\mathcal{I}^{\kappa}_{s'_n}$ is a controlled subfamily of indexes in the sense of Definition \ref{def:contrsub}. Without loss of generality we can also assume that
\[
\lim_{n\rightarrow +\infty} \frac{s_n'}{\e_n \log\left(\e_n^{-d}\right)^{\sfrac{1}{d}}}=\beta_0<+\infty.
\]
Define $u_n:\mathcal{I}_{s'_n}(Q)\setminus\mathcal{J}_n \rightarrow \R_+$
    \begin{equation*}
        u_n(J): =
            \frac{1}{\eta_{\e_n}(Q_{s_n'}(J))}\sum_{x\in \eta_{\e_n}\cap Q_{s'_n}(J)} u_{\e_n}(x)  \ \ \ \text{for all $J\in\mathcal{I}_{s'_n}(Q)\setminus \mathcal{J}_n$}.
    \end{equation*}

If $J,J'\in \mathcal{I}_{s'_n}(Q)\setminus\mathcal{J}_n$, $J'\in N_{\e_n}(J)$ then we have
\begin{align*}
|u_n(J)-u_n(J')|^p (s_n')^{d-p}\leq& C \frac{\e_n^{2d}}{s_n^{2d}}\sum_{x\in \eta_{\e_n}\cap  Q_{s'_n}(J)}\sum_{y\in \eta_{\e_n}\cap  Q_{3s'_n}(J)}|u_{\e_n}(x)-u_{\e_n}(y)|^p s_n^{d-p}\\
\leq& C   \e_n^d\sum_{x\in \eta_{\e_n}\cap  Q_{s'_n}(J)} |\mathrm{grad}_{s_n}u_{\e_n}(x)|^p
\end{align*}
 where the prefactor  $\frac{\e_n^{d}}{s_n^{d}}$  appears because of double-counting and where we exploited \eqref{sqr}. Henceforth 
\begin{equation}\label{energycontrol}
\sum_{\substack{J,J'\in \mathcal{I}_{s'_n}(Q)\setminus \mathcal{J}_n :\\ J'\in N_{\e_n}(J)}}|u_n(J)-u_n(J')|^p (s_n')^{d-p} <C\mathcal{F}_{\e_n}(u_{\e_n};A).
\end{equation}
In the same way
\begin{equation}\label{normcontrol}
\sum_{\substack{J\in \mathcal{I}_{s'_n}(Q)\setminus  \mathcal{J}_n } } |u_n(J)|^p (s_n')^{d}\leq C \sum_{x\in\eta_{\e_n}\cap Q}|u_{\e_n}(x)|^p \e_n^d.
\end{equation}
By invoking the extension Lemma \ref{lem:extension} we can find a $Tu_n:\mathcal{I}_{s'_n}(Q)\rightarrow \R_+$ which has (thanks to Lemma \ref{lem:alici} ) piecewise constant extension $\widehat{Tu_n}$ precompact in $L^q$ for each $q<p$. In particular, due to Sobolev embeddings we conclude that $\widehat{Tu_n}\rightarrow u$ in $L^p$ where $u\in W^{1,q}(Q)$ for all $q<p$. Notice that (recall the notation  \ref{eqn:index}, \ref{eqn:set}) by weak compactness we also infer that
    \[
    (\nabla \widehat{Tu_n})\ca_{Q\setminus E_{s_n'}^{\kappa}}\stackrel{L^p}{\rightharpoonup} V, \ \ \text{for some $V\in L^p$}.
    \]
Observe that 
\[
|E_{s'_n}^{\kappa}|=\#(\mathcal{I}_{s_n'}^{\kappa}(Q)) (s'_n)^d 
\]
If we are in situation a) of Proposition \ref{prop:decay} (namely $\beta_0>\gamma^{-\sfrac{1}{d}}$) then $|E^{\kappa}_{s_n'}|=0$ for $n$ big enough. If instead case b) is in force then
\[
|E_{s'_n}^{\kappa}|=\#(\mathcal{I}_{s'_n}^{\kappa}(Q)) (s'_n)^d \leq  \e_n^{\varrho_0} \rightarrow 0.
\]
In both cases $|E_{s'_n}^{\kappa}|\rightarrow 0 $. Thus $ \ca_{Q\setminus E_{s_n'}^{\kappa}}\rightarrow \ca_Q$ strongly in $L^q$ and $(\nabla \widehat{Tu_n})\stackrel{L^q}{\rightharpoonup} \nabla u$ then 
 
    \[
    (\nabla \widehat{Tu_n})\ca_{Q\setminus E_{s_n'}^{\kappa}}\stackrel{L^q}{\rightharpoonup} \nabla u
    \]
 yielding that $\nabla u=V\in L^p$. In particular $u\in W^{1,p}(\Omega)$ and $\widehat{Tu_n}\rightarrow u$ in $L^p(Q)$. Also
    \begin{align}
        \int_{\mathcal{V}(\eta^{\alpha}_{\e})\cap Q\setminus E_{s'_n}^{\kappa}}|\hat{u}_{\e_n}(x)-\widehat{Tu_n}(x)|^p&= \sum_{J\in \mathcal{I}_{s'_n}(Q)\setminus \mathcal{I}_{s'_n}^\kappa(Q)}  \int_{\mathcal{V}(\eta^{\alpha}_{\e})\cap Q_{s'_n}(J)} |\hat{u}_{\e_n}(x)-u_n(J)|^p \d x\nonumber\\
        &\leq C\e_n^d \sum_{J\in \mathcal{I}_{s'_n}(Q)\setminus \mathcal{I}_{s'_n}^\kappa(Q)} \sum_{x\in \eta^{\alpha}_{\e}\cap Q_{s'_n}(J)}  |u_{\e_n}(x)-u_n(J)|^p \nonumber\\
               &\leq C\frac{\e_n^{2d}}{s_n^d} \sum_{J\in \mathcal{I}_{s'_n}(Q)\setminus \mathcal{I}_{s'_n}^\kappa(Q)} \sum_{x,y\in \eta^{\alpha}_{\e}\cap Q_{s'_n}(J)}  |u_{\e_n}(x)-u_{\e_n}(y)|^p \nonumber\\
              &\leq Cs_n^p \F_{\e_n}(u_{\e_n};Q)\rightarrow 0.\label{eqn;uno}
    \end{align}
Hence
\begin{align}
   \int_{\mathcal{V}_{\e_n}(\eta^{\alpha}_{\e_n})\cap   E_{s'_n}^{\kappa}}&|\hat{u}_{\e_n}(x)-\widehat{Tu_n}(x)|^q\d x \nonumber\\
   &\leq    \left(\int_{\mathcal{V}_{\e_n}(\eta^{\alpha}_{\e_n})\cap   E_{s'_n}^{\kappa}}|\hat{u}_{\e_n}(x)-\widehat{Tu_n}(x)|^p\d x\right)^{\sfrac{q}{p}}|E_{s'_n}^{\kappa}|^{\frac{(p-q)}{p}} \rightarrow 0 \label{eqn;deu}
\end{align}
for all $q<p$. \rc{Observe that, due to the property \eqref{eqn:bndLimit2} of the extended function $Tu_n$, to the bound \eqref{eqn:bndCMP} and to the very definition of $\eta_{\a}$ we can infer 
\begin{equation}\label{eqn:refcor}
 \left(\int_{\mathcal{V}_{\e_n}(\eta^{\alpha}_{\e_n})\cap   E_{s'_n}^{\kappa}}|\hat{u}_{\e_n}(x)-\widehat{Tu_n}(x)|^p\d x\right)^{\sfrac{q}{p}}<C
\end{equation}
 }
By collecting \eqref{eqn;uno}, \eqref{eqn;deu} and \eqref{eqn:refcor} we conclude that $u_{\e_n}\stackrel{q}{\longrightarrow} u\in W^{1,p}(Q)$ in the sense of Definition \ref{def:conv} for all $q<p$.
\end{proof}

\bibliography{references}
\bibliographystyle{plain}

\end{document}